\documentclass{article}

\usepackage[english]{babel}

\usepackage[letterpaper,top=2cm,bottom=2cm,left=2.85cm,right=2.85cm,marginparwidth=1.75cm]{geometry}

\usepackage{amsmath}
\usepackage{amssymb}
\usepackage{amsthm}
\usepackage{mathtools}
\usepackage{cite}
\usepackage{tabstackengine}
\stackMath
\makeatletter
\renewcommand\TAB@delim[1]{\scriptstyle#1}
\makeatother
\setstackgap{S}{2pt}
\usepackage{xcolor}
\usepackage{graphicx}
\usepackage[colorlinks=true, hidelinks]{hyperref}
\newtheorem{theorem}{Theorem}[section]
\newtheorem{lemma}[theorem]{Lemma}
\newtheorem{proposition}[theorem]{Proposition}
\newtheorem{corollary}[theorem]{Corollary}

\begin{document}
\begin{center}
		\vskip 1cm{\LARGE \bf Integral Representations for Multiple Ap\'ery-Like Series}\\
		\vskip 5mm
		Jorge Antonio Gonz\'alez Layja\footnotemark\\
        Mexico
	\end{center}
\footnotetext{Email: \href{mailto:jorgelayja16@gmail.com}{jorgelayja16@gmail.com}}
\begin{abstract}
We derive integral representations for six families of multiple Ap\'ery-like series using repeated integration by parts and Fourier expansions. The resulting formulas are expressed in terms of polylogarithms, Legendre chi functions, and inverse tangent integrals. As applications, we recover several known evaluations as special cases of our results, expressed in terms of Dirichlet eta, beta, and lambda functions. In addition, we obtain a new identity expressing a family of such series as linear combinations of products of Dirichlet eta values.
\end{abstract}

\\[1ex]

\section{Introduction}
Series involving the central binomial coefficient $\binom{2n}{n}$, referred to as \emph{Ap\'ery‑like series}, have attracted sustained attention due to their connections with special functions, most notably the Riemann zeta function $\zeta$, as well as their evaluations in terms of classical constants, as illustrated in \cite{preprint3,journal5,preprint1,journal9,journal11,journal10,journal7,preprint2,journal8,journal6,journal12}. This terminology originates from Ap\'ery’s \cite{journal4} celebrated proof of the irrationality of $\zeta (2)$ and $\zeta (3)$, in which the series representations
$$\zeta (2)=3\sum _{n=1}^{\infty }\frac{1}{n^2\binom{2n}{n}},\qquad \zeta (3)=\frac{5}{2}\sum _{n=1}^{\infty }\frac{(-1)^{n-1}}{n^3\binom{2n}{n}},$$
play a fundamental role. When such series also involve nested harmonic sums, they are called \emph{multiple Ap\'ery‑like series}, which form the focus of the present work. A notable recent contribution in this direction is due to Gen\v{c}ev and Rucki \cite{journal2}, who established evaluations for several such series in terms of special values of Dirichlet $L$-functions. In particular, their paper concludes with a conjecture \cite[Eq.~(7.2)]{journal2}, subsequently proved by Xu \cite[Thm.~2.2]{preprint1} via two independent hypergeometric methods, which asserts that for $j\in \mathbb{Z}_{\ge 0}$,
$$\sum _{n=1}^{\infty }\frac{\binom{2n}{n}}{n\,4^n}\sum _{n\ge n_1\ge \cdots \ge n_j\ge 1}\prod _{i=1}^j\frac{1}{n_i^2}=2\,\eta (2j+1),$$
where $\eta$ denotes the Dirichlet eta function. A second related result \cite[Thm.~7.1]{journal2}, which was originally conditional on the preceding conjecture, states that for $j\in \mathbb{Z}_{>0}$, a modified multiple Ap\'ery-like series can be expressed in terms of the Dirichlet lambda function $\lambda$, namely
$$\sum _{n=1}^{\infty }\frac{\binom{2n}{n}}{n\,4^n}\sum _{n\ge n_1\ge \cdots \ge n_j\ge 1}\frac{4^{n_j}}{\binom{2n_j}{n_j}}\prod _{i=1}^j\frac{1}{n_i^2}=4\,\lambda (2j+1).$$
In this paper, we employ a systematic approach that recovers both of these results and yields further explicit evaluations, including a new identity expressing a family of such series as a finite alternating sum of products of Dirichlet eta values. The method is based on iterated applications of integration by parts to suitable trigonometric integrals, which generate recursive relations leading to nested harmonic sums, together with Fourier expansions involving the polylogarithm and the Legendre chi function that allow such integrals to be expressed in terms of special functions.

More precisely, for $j\in \mathbb{Z}_{\ge 0}$ and $k\in \mathbb{Z}_{>0}$, we establish integral representations for
$$\sum _{n=1}^{\infty }\frac{\binom{2n}{n}}{n^k4^n}\zeta _n^{\star }(\{2\}_j),\qquad\sum _{n=1}^{\infty }\frac{4^n}{n^{k+1}\binom{2n}{n}}t_n^{\star }(\{2\}_j),$$
along with a shifted $(2n+1)$-indexed analogue of the second family. These correspond to Theorems \ref{thm3.1}, \ref{thm3.4}, and \ref{thm3.9}.

For $j\in \mathbb{Z}_{>0}$, we also obtain formulas for the related classes
$$\sum _{n=1}^{\infty }\frac{\binom{2n}{n}}{n^k4^n}\sum _{n\ge n_1\ge \cdots \ge n_j\ge 1}\frac{4^{n_j}}{\binom{2n_j}{n_j}}\prod _{i=1}^j\frac{1}{n_i^2},$$
$$\sum _{n=1}^{\infty }\frac{4^n}{n^{k+1}\binom{2n}{n}}\sum _{n\ge n_1\ge \cdots \ge n_{j-1}>n_j\ge 0}\frac{\binom{2n_j}{n_j}}{(2n_j+1)4^{n_j}}\prod _{i=1}^{j-1}\frac{1}{(2n_i-1)^2},$$
as well as a $(2n+1)$-indexed variant of the latter. These are given in Theorems \ref{thm3.6}, \ref{thm3.8}, and \ref{thm3.11}.

Altogether, this yields six distinct integral representations, expressed in terms of polylogarithms, Legendre chi functions, and inverse tangent integrals. Here, $\zeta_n^{\star}(\{2\}_j)$ and $t_n^{\star }(\{2\}_j)$ denote the multiple harmonic star sums of depth $j$ and weight $2j$, and their odd analogues, defined by
$$\zeta_n^{\star}(\{2\}_j)\coloneqq\sum _{n\ge n_1\ge \cdots \ge n_j\ge 1}\prod _{i=1}^j\frac{1}{n_i^2},\qquad t_n^{\star}(\{2\}_j)\coloneqq\sum _{n\ge n_1\ge \cdots \ge n_j\ge 1}\prod _{i=1}^j\frac{1}{(2n_i-1)^2},$$
with the convention $\zeta_n^{\star}(\emptyset )=t_n^{\star}(\emptyset )\coloneqq1$. These are finite versions of the multiple zeta star values and multiple $t$-star values introduced by Hoffman in \cite{journal3,journal1}, respectively.

For completeness, we recall the special functions appearing in our results. The polylogarithm of order $s$ is defined by
$$
\operatorname{Li}_s(x)\coloneqq\sum _{n=1}^{\infty }\frac{x^n}{n^s},
\qquad
\begin{cases}
|x| < 1, & s \in \mathbb{C}, \\
x = 1, & \mathfrak{R}(s) > 1, \\
x = -1, & \mathfrak{R}(s) > 0.
\end{cases}
$$
In addition, $\operatorname{Li}_s(1)=\zeta (s)$ and $\operatorname{Li}_s(-1)=-\eta (s)$, where $\zeta$ and $\eta$ are the Riemann zeta and Dirichlet eta functions (see \cite[p.~189]{book4}). It satisfies
$$\operatorname{Li}_s(x)=\int _0^x\frac{\operatorname{Li}_{s-1}(t)}{t}\,dt.$$
The Legendre chi function is defined by
$$
\chi _s(x)\coloneqq\sum _{n=1}^{\infty }\frac{x^{2n-1}}{(2n-1)^s},
\qquad
\begin{cases}
|x| < 1, & s \in \mathbb{C}, \\
x = \pm 1, & \mathfrak{R}(s) > 1,
\end{cases}
$$
with $\chi _1(x)=\operatorname{arctanh} (x)$ and $\chi _s(1)=\lambda (s)$, where $\lambda $ is the Dirichlet lambda function (see \cite[p.~189]{book4}). It admits the representation $\chi _s(x)=\frac{1}{2}\left(\operatorname{Li}_s(x)-\operatorname{Li}_s(-x)\right)$.

For $m\in \mathbb{Z}_{> 0}$, the inverse tangent integrals are defined by
$$\operatorname{Ti}_m(x)\coloneqq\sum _{n=1}^{\infty }\frac{(-1)^{n-1}x^{2n-1}}{(2n-1)^m},\qquad |x|\le 1,$$
with $\operatorname{Ti}_1(x)=\arctan (x)$ and $\operatorname{Ti}_m(1)=\beta (m)$, where $\beta $ denotes the Dirichlet beta function (see \cite[p.~190]{book4}).

Similarly to the polylogarithm, the Legendre chi function and the inverse tangent integrals satisfy
$$\chi _s(x)=\int _0^x\frac{\chi _{s-1}(t)}{t}\,dt,\qquad \operatorname{Ti}_m(x)=\int _0^x\frac{\operatorname{Ti}_{m-1}(t)}{t}\,dt.$$
\newpage
\section{Lemmas}
In this section, we collect auxiliary lemmas that underpin the proofs of the main results. Lemmas \ref{lma2.1} and \ref{lma2.2} relate trigonometric integrals to finite sums involving powers of $\frac{\pi}{2}$, odd harmonic numbers, and nested harmonic sums. Lemma \ref{lma2.3} provides Fourier expansions for $\operatorname{Li}_k(\cos ^2(x))$, $\operatorname{Li}_k(\sin ^2(x))$, $\chi _k(\cos (x))$, and $\chi _k(\sin (x))$ with coefficients given by definite integrals involving $\ln ^{k-1}\!\left(\frac{1+t}{2\sqrt{t}}\right)$ and $\ln ^{k-1}\!\left(\frac{1+t^2}{2t}\right)$.

\begin{lemma} \label{lma2.1}The following identities hold:
\begin{alignat*}{2}
(\mathrm{i})&\quad&&\text{For }m\in \mathbb{Z}_{\ge 0},n\in \mathbb{Z}_{>0},\\
&&&\int _0^{\frac{\pi }{2}}x^{2m}\left((-1)^{n-1}+\cos (2nx)\right)\,dx=(2m)!\sum _{j=0}^m\frac{(-1)^{j+n-1}}{(2n)^{2j}(2m-2j+1)!}\left(\frac{\pi }{2}\right)^{2m-2j+1}.\\
(\mathrm{ii})&\quad&&\text{For }m,n\in \mathbb{Z}_{>0},\\
&&&\int _0^{\frac{\pi }{2}}x^{2m-1}\left((-1)^{n-1}+\cos (2nx)\right)\,dx\\
&&&=(2m-1)!\left(\sum _{j=0}^{m-1}\frac{(-1)^{j+n-1}}{(2n)^{2j}(2m-2j)!}\left(\frac{\pi }{2}\right)^{2m-2j}+(-1)^m\frac{1-(-1)^n}{(2n)^{2m}}\right).\\
(\mathrm{iii})&\quad&&\text{For }m,n\in \mathbb{Z}_{\ge 0},\\
&&&\int _0^{\frac{\pi }{2}}x^{2m}\frac{(-1)^{n-1}+\cos (2nx)}{\cos (x)}\,dx=2(2m)!\sum _{j=0}^m\frac{(-1)^{j+n-1}O_n^{(2j+1)}}{(2m-2j)!}\left(\frac{\pi }{2}\right)^{2m-2j}.\\
(\mathrm{iv})&\quad&&\text{For }m\in \mathbb{Z}_{>0}, n\in \mathbb{Z}_{\ge 0},\\
&&&\int _0^{\frac{\pi }{2}}x^{2m-1}\frac{1-\cos (2nx)}{\sin (x)}\,dx=2(2m-1)!\sum _{j=1}^m\frac{(-1)^{j-1}\overline{O}_n^{(2j)}}{(2m-2j)!}\left(\frac{\pi }{2}\right)^{2m-2j},
\end{alignat*}
where $O_n^{(m)}\coloneqq\sum _{k=1}^n\frac{1}{(2k-1)^m}$ and $\overline{O}_n^{(m)}\coloneqq\sum _{k=1}^n\frac{(-1)^{k-1}}{(2k-1)^m}$ denote the odd harmonic numbers of order $m$ and their alternating analogues, respectively.
\end{lemma}
\begin{proof}
Each part of this lemma follows by repeated integration by parts, which produces finite sums that can be rearranged into the stated forms. For example, consider
$$I_{m,n}=\int _0^{\frac{\pi }{2}}x^{2m}\cos (2nx)\,dx.$$
Iterating integration by parts four times and noting that all boundary terms involving $\sin (2nx)$ vanish at the endpoints, we obtain
\begin{align*}
I_{m,n}&=\left[\frac{\cos (2nx)}{(2n)^2}\left(x^{2m}\right)'\right]^{\frac{\pi }{2}}_0-\frac{1}{(2n)^2}\int _0^{\frac{\pi }{2}}\left(x^{2m}\right)''\cos (2nx)\,dx.\\
&=\left[\frac{\cos (2nx)}{(2n)^2}\left(x^{2m}\right)'-\frac{\cos (2nx)}{(2n)^4}\left(x^{2m}\right)'''\right]^{\frac{\pi }{2}}_0+\frac{1}{(2n)^4}\int _0^{\frac{\pi }{2}}\left(x^{2m}\right)^{(4)}\cos (2nx)\,dx.
\end{align*}
After $2m$ iterations, it follows that
\begin{align*}
I_{m,n}&=\sum _{j=1}^m\frac{(-1)^{j-1}}{(2n)^{2j}}\left[\cos (2nx)\left(x^{2m}\right)^{(2j-1)}\right]^{\frac{\pi }{2}}_0+\frac{(-1)^m}{(2n)^{2m}}\int _0^{\frac{\pi }{2}}\left(x^{2m}\right)^{(2m)}\cos (2nx)\,dx\\
&=(2m)!\sum _{j=1}^m\frac{(-1)^{j+n-1}}{(2n)^{2j}(2m-2j+1)!}\left(\frac{\pi }{2}\right)^{2m-2j+1}\\
&=(2m)!\sum _{j=0}^m\frac{(-1)^{j+n-1}}{(2n)^{2j}(2m-2j+1)!}\left(\frac{\pi }{2}\right)^{2m-2j+1}-\frac{(-1)^{n-1}}{2m+1}\left(\frac{\pi }{2}\right)^{2m+1}.
\end{align*}
Combining this with $\int _0^{\frac{\pi }{2}}x^{2m}\,dx=\frac{1}{2m+1}\left(\frac{\pi }{2}\right)^{2m+1}$ and rearranging gives the identity in part (i).

The remaining parts follow by the same procedure. For parts (iii) and (iv), we additionally use the identities (see \cite[1.342]{book3})
\begin{align*}
&2\sum _{k=1}^n(-1)^{n+k}\cos ((2k-1)x)=\frac{(-1)^{n-1}+\cos (2nx)}{\cos (x)},\\
&2\sum _{k=1}^n\sin ((2k-1)x)=\frac{1-\cos (2nx)}{\sin (x)}.
\end{align*}
\end{proof}

\begin{lemma} \label{lma2.2}The following identities hold:
\begin{alignat*}{2}
(\mathrm{i})&\quad&&\text{For }m\in \mathbb{Z}_{\ge 0},n\in \mathbb{Z}_{>0},\\
&&&\int _0^{\frac{\pi }{2}}x^{2m}\cos ^{2n}(x)\,dx=(2m)!\frac{\binom{2n}{n}}{4^n}\sum _{j=0}^m\frac{(-1)^j\zeta_n^{\star}(\{2\}_j)}{2^{2j}(2m-2j+1)!}\left(\frac{\pi }{2}\right)^{2m-2j+1}.\\[1ex]
(\mathrm{ii})&\quad&&\text{For }m\in \mathbb{Z}_{\ge 0},n\in \mathbb{Z}_{>0},\\
&&&\int _0^{\frac{\pi }{2}}x^{2m}\cos ^{2n-1}(x)\,dx=\frac{(2m)!}{2}\frac{4^n}{n\binom{2n}{n}}\sum _{j=0}^m\frac{(-1)^jt_n^{\star}(\{2\}_j)}{(2m-2j)!}\left(\frac{\pi }{2}\right)^{2m-2j}.\\[1ex]
(\mathrm{iii})&\quad&&\text{For }m,n\in \mathbb{Z}_{>0},\\
&&&\int _0^{\frac{\pi }{2}}x^{2m-1}\cos ^{2n}(x)\,dx\\
&&&=(2m-1)!\frac{\binom{2n}{n}}{4^n}\left(\sum _{j=0}^{m-1}\frac{(-1)^j\zeta _n^{\star }(\{2\}_j)}{2^{2j}(2m-2j)!}\left(\frac{\pi }{2}\right)^{2m-2j}+\frac{(-1)^m}{4^m}\sum _{n\ge n_1\ge \cdots \ge n_m\ge 1}\frac{4^{n_m}}{\binom{2n_m}{n_m}}\prod _{i=1}^m\frac{1}{n_i^2}\right).\\[1ex]
(\mathrm{iv})&\quad&&\text{For }m,n\in \mathbb{Z}_{>0},\\
&&&\int _0^{\frac{\pi }{2}}x^{2m-1}\sin ^{2n-1}(x)\,dx\\
&&&=\frac{(2m-1)!}{2}\frac{4^n}{n\binom{2n}{n}}\sum _{j=1}^m\frac{(-1)^{j-1}}{(2m-2j)!}\left(\frac{\pi }{2}\right)^{2m-2j}\sum _{n\ge n_1\ge \cdots \ge n_{j-1}>n_j\ge 0}\frac{\binom{2n_j}{n_j}}{(2n_j+1)4^{n_j}}\prod _{i=1}^{j-1}\frac{1}{(2n_i-1)^2}.
\end{alignat*}
\end{lemma}
\begin{proof}
All parts of this lemma are proved in the same manner. A recurrence relation is first obtained via integration by parts. Evaluation of the first few cases suggests the general form of the identities in terms of multiple harmonic sums. This form is then verified by induction on $m$, using the recurrence, which preserves the structure of the resulting expressions. We illustrate this strategy by proving the identity in part (i).

Let
$$I_{m,n_1}=\int _0^{\frac{\pi }{2}}x^{2m}\cos ^{2n_1}(x)\,dx.$$
Applying integration by parts yields
\begin{align*}
I_{m,n_1}&=\int _0^{\frac{\pi }{2}}x^{2m}(\sin (x))'\cos ^{2n_1-1}(x)\,dx\\
&=\frac{m}{n_1}\int _0^{\frac{\pi }{2}}x^{2m-1}\left(\cos ^{2n_1}(x)\right)'\,dx+(2n_1-1)\int _0^{\frac{\pi }{2}}x^{2m}\sin ^2(x)\cos ^{2n_1-2}(x)\,dx\\
&=-\frac{m(2m-1)}{n_1}I_{m-1,n_1}+(2n_1-1)I_{m,n_1-1}-(2n_1-1)I_{m,n_1}.
\end{align*}
From this we obtain
$$\frac{1}{2}\frac{m(2m-1)}{n_1^2}I_{m-1,n_1}=\frac{2n_1-1}{2n_1}I_{m,n_1-1}-I_{m,n_1}.$$
By multiplying both sides by $\frac{4^{n_1}}{\binom{2n_1}{n_1}}$, using the identity $\binom{2n_1}{n_1}=\frac{2\left(2n_1-1\right)}{n_1}\binom{2n_1-2}{n_1-1}$, and summing from $n_1=1$ to $n$, we further obtain
\begin{align*}
\frac{m(2m-1)}{2}\sum _{n_1=1}^n\frac{4^{n_1}}{n_1^2\binom{2n_1}{n_1}}I_{m-1,n_1}&=\sum _{n_1=1}^n\left(\frac{4^{n_1-1}}{\binom{2n_1-2}{n_1-1}}I_{m,n_1-1}-\frac{4^{n_1}}{\binom{2n_1}{n_1}}I_{m,n_1}\right)\\
&=I_{m,0}-\frac{4^n}{\binom{2n}{n}}I_{m,n}.
\end{align*}
Since $I_{m,0}=\frac{1}{2m+1}\left(\frac{\pi }{2}\right)^{2m+1}$, rearranging gives the recurrence
$$I_{m,n}=\frac{\binom{2n}{n}}{4^n}\left(\frac{1}{2m+1}\left(\frac{\pi }{2}\right)^{2m+1}-\frac{(2m)(2m-1)}{4}\sum _{n_1=1}^n\frac{4^{n_1}}{n_1^2\binom{2n_1}{n_1}}I_{m-1,n_1}\right).$$
Iterating the recurrence for the first few values of $m$ yields
\begin{align*}
I_{0,n}&=\frac{\binom{2n}{n}}{4^n}\left(\frac{\pi }{2}\right),\\
I_{1,n}&=\frac{\binom{2n}{n}}{4^n}\left(\frac{1}{3}\left(\frac{\pi }{2}\right)^3-\frac{1}{2}\left(\frac{\pi }{2}\right)\sum _{n_1=1}^n\frac{1}{n_1^2}\right),\\
I_{2,n}&=\frac{\binom{2n}{n}}{4^n}\left(\frac{1}{5}\left(\frac{\pi }{2}\right)^5-\left(\frac{\pi }{2}\right)^3\sum _{n_1=1}^n\frac{1}{n_1^2}+\frac{3}{2}\left(\frac{\pi }{2}\right)\sum _{n_1=1}^n\frac{1}{n_1^2}\sum _{n_2=1}^{n_1}\frac{1}{n_2^2}\right).
\end{align*}
These initial cases suggest the stated formula in part (i). Since the recurrence preserves this structure, the identity is verified by induction on $m$.

The remaining parts are handled by the same argument. Integration by parts yields analogous recurrences, and the identities follow by induction on $m$.
\end{proof}

\begin{lemma} \label{lma2.3}Let $k\in \mathbb{Z}_{>1}, x\in \mathbb{R}$. Then
\begin{alignat*}{2}
\left(\operatorname{i}\right)&\quad&&\operatorname{Li}_k(\cos ^2(x))=\frac{2^k}{(k-1)!}\sum _{n=1}^{\infty }\left(\int _0^1t^{n-1}\ln ^{k-1}\!\left(\frac{1+t}{2\sqrt{t}}\right)\,dt\right)\left((-1)^{n-1}+\cos (2nx)\right),\\
\left(\operatorname{ii}\right)&\quad&&\operatorname{Li}_k(\sin ^2(x))=\frac{2^k}{(k-1)!}\sum _{n=1}^{\infty }(-1)^{n-1}\left(\int _0^1t^{n-1}\ln ^{k-1}\!\left(\frac{1+t}{2\sqrt{t}}\right)\,dt\right)\left(1-\cos (2nx)\right),\\
\left(\operatorname{iii}\right)&\quad&&\chi _k(\cos (x))=\frac{2}{(k-1)!}\sum _{n=1}^{\infty }\left(\int _0^1t^{2n-2}\ln ^{k-1}\!\left(\frac{1+t^2}{2t}\right)\,dt\right)\cos ((2n-1)x),\\
\left(\operatorname{iv}\right)&\quad&&\chi _k(\sin (x))=\frac{2}{(k-1)!}\sum _{n=1}^{\infty }(-1)^{n-1}\left(\int _0^1t^{2n-2}\ln ^{k-1}\!\left(\frac{1+t^2}{2t}\right)\,dt\right)\sin ((2n-1)x).
\end{alignat*}
\end{lemma}
\begin{proof}
The approach employed in this lemma follows V\u{a}lean's strategy in \cite[Sect.~6.49]{book1}. For part (i), we consider the identity (see \cite[1.461.2]{book3})
$$\sum _{n=1}^{\infty }\left((-1)^{n-1}+\cos (2nx)\right)e^{-2ny}=-\frac{\cos ^2(x)\tanh (y)}{\cos (2x)-\cosh (2y)}.$$
On the other hand, by substituting $t=e^{-2y}$, we have
$$\frac{1}{2}\int _0^1t^{n-1}\ln ^{k-1}\!\left(\frac{1+t}{2\sqrt{t}}\right)\,dt=\int _0^{\infty }\ln ^{k-1}(\cosh (y))e^{-2ny}\,dy.$$
Therefore, multiplying both sides by $(-1)^{n-1}+\cos (2nx)$ and summing over $n\ge 1$ gives
\begin{align*}
\frac{1}{2}&\sum _{n=1}^{\infty }\left(\int _0^1t^{n-1}\ln ^{k-1}\!\left(\frac{1+t}{2\sqrt{t}}\right)\,dt\right)\left((-1)^{n-1}+\cos (2nx)\right)\\
&=\int _0^{\infty }\ln ^{k-1}\!\left(\cosh (y)\right)\left(\sum _{n=1}^{\infty }\left((-1)^{n-1}+\cos (2nx)\right)e^{-2ny}\right)\,dy\\
&=-\int _0^{\infty }\frac{\cos ^2(x)\tanh (y)\ln ^{k-1}\!\left(\cosh (y)\right)}{\cos (2x)-\cosh (2y)}\,dy\\
&\overset{\cosh (y)=\frac{1}{t}}{=}\frac{(-1)^{k-1}}{2^{k+1}}\int _0^1\frac{\cos ^2(x)\ln ^{k-1}(t)}{1-\cos ^2(x)t}\,dt.
\end{align*}
By using the identity (see \cite[Sect.~1.6]{book5})
\begin{equation*}
\int _0^1\frac{x\ln ^n(t)}{1-xt}\,dt=(-1)^nn!\operatorname{Li}_{n+1}(x),\label{eq2.1}\tag{2.1}
\end{equation*}
with $x\mapsto \cos ^2(x)$ and $n\mapsto k-1$, the result in part (i) follows. The identity in part (ii) then follows from part (i) by the substitution $x\mapsto \frac{\pi}{2}-x$.

For part (iii), we proceed similarly. Using the identity (see \cite[1.461.2]{book3})
$$\sum _{n=1}^{\infty }\cos ((2n-1)x)e^{-(2n-1)y}=-\frac{\cos (x)\sinh (y)}{\cos (2x)-\cosh (2y)},$$
together with
$$\int _0^1t^{2n-2}\ln ^{k-1}\!\left(\frac{1+t^2}{2t}\right)\,dt\overset{t=e^{-y}}{=}\int _0^{\infty }\ln ^{k-1}(\cosh (y))e^{-(2n-1)y}\,dy,$$
we deduce
\begin{align*}
&\sum _{n=1}^{\infty }\left(\int _0^1t^{2n-2}\ln ^{k-1}\!\left(\frac{1+t^2}{2t}\right)\,dt\right)\cos ((2n-1)x)\\
&=\int _0^{\infty }\ln ^{k-1}(\cosh (y))\left(\sum _{n=1}^{\infty }\cos ((2n-1)x)e^{-(2n-1)y}\right)\,dy\\
&=-\int _0^{\infty }\frac{\cos (x)\sinh (y)\ln ^{k-1}(\cosh (y))}{\cos (2x)-\cosh (2y)}\,dy\\
&\overset{\cosh (y)=\frac{1}{t}}{=}\frac{(-1)^{k-1}}{2}\int _0^1\frac{\cos (x)\ln ^{k-1}(t)}{1-\cos ^2(x)t^2}\,dt.
\end{align*}
By expanding the resulting integral and employing \eqref{eq2.1} and $\chi _s(x)=\frac{1}{2}\left(\operatorname{Li}_s(x)-\operatorname{Li}_s(-x)\right)$, we obtain the result in part (iii). Substituting $x\mapsto \frac{\pi}{2}-x$ in part (iii) gives part (iv).
\end{proof}

\section{Main Results}
In this section, we present the main integral representations for the multiple Ap\'ery‑like series introduced above, together with their consequences. Theorems \ref{thm3.1}, \ref{thm3.4}, \ref{thm3.6}, \ref{thm3.8}, \ref{thm3.9}, and \ref{thm3.11} establish these representations, while Corollaries \ref{corollary3.2} and \ref{corollary3.7} recover the results of Gen\v{c}ev and Rucki introduced earlier. Proposition \ref{prop3.3} provides a new evaluation, and Proposition \ref{prop3.5} and Corollaries \ref{corollary3.10} and \ref{corollary3.12} recover additional known identities.

\begin{theorem} \label{thm3.1}Let $j\in \mathbb{Z}_{\ge 0}, k\in \mathbb{Z}_{>0}$. Then
$$\sum _{n=1}^{\infty }\frac{\binom{2n}{n}}{n^k4^n}\zeta _n^{\star }(\{2\}_j)=-\frac{2^k}{(k-1)!}\int _0^1\frac{\ln ^{k-1}\!\left(\frac{1+t}{2\sqrt{t}}\right)\operatorname{Li}_{2j}(-t)}{t}\,dt.$$
\end{theorem}
\begin{proof}
We begin by noting that
$$\sum _{n=1}^{\infty }\frac{1}{n^k}\int _0^{\frac{\pi }{2}}x^{2m}\cos ^{2n}(x)\,dx=\int _0^{\frac{\pi }{2}}x^{2m}\operatorname{Li}_k(\cos ^2(x))\,dx.$$
Applying part (i) of Lemmas \ref{lma2.2} and \ref{lma2.3} yields
\begin{align*}
&(2m)!\sum _{n=1}^{\infty }\frac{\binom{2n}{n}}{n^k4^n}\sum _{j=0}^m\frac{(-1)^j\zeta _n^{\star }(\{2\}_j)}{2^{2j}(2m-2j+1)!}\left(\frac{\pi }{2}\right)^{2m-2j+1}\\
&=\frac{2^k}{(k-1)!}\sum _{n=1}^{\infty }\left(\int _0^1t^{n-1}\ln ^{k-1}\!\left(\frac{1+t}{2\sqrt{t}}\right)\,dt\right)\int _0^{\frac{\pi }{2}}x^{2m}\left((-1)^{n-1}+\cos (2nx)\right)\,dx.
\end{align*}
By part (i) of Lemma \ref{lma2.1}, we obtain
\begin{align*}
&(2m)!\sum _{j=0}^m\frac{(-1)^j}{2^{2j}(2m-2j+1)!}\left(\frac{\pi }{2}\right)^{2m-2j+1}\sum _{n=1}^{\infty }\frac{\binom{2n}{n}}{n^k4^n}\zeta _n^{\star }(\{2\}_j)\\
&=(2m)!\frac{2^k}{(k-1)!}\sum _{n=1}^{\infty }\left(\int _0^1t^{n-1}\ln ^{k-1}\!\left(\frac{1+t}{2\sqrt{t}}\right)\,dt\right)\sum _{j=0}^m\frac{(-1)^{j+n-1}}{(2n)^{2j}(2m-2j+1)!}\left(\frac{\pi }{2}\right)^{2m-2j+1}\\
&=(2m)!\sum _{j=0}^m\frac{(-1)^{j-1}}{2^{2j}(2m-2j+1)!}\left(\frac{\pi }{2}\right)^{2m-2j+1}\frac{2^k}{(k-1)!}\int _0^1\left(\sum _{n=1}^{\infty }\frac{(-t)^n}{n^{2j}}\right)\frac{\ln ^{k-1}\!\left(\frac{1+t}{2\sqrt{t}}\right)}{t}\,dt.
\end{align*}
Therefore,
\begin{align*}
&\sum _{j=0}^m\frac{(-1)^j}{2^{2j}(2m-2j+1)!}\left(\frac{\pi }{2}\right)^{2m-2j+1}\left(\sum _{n=1}^{\infty }\frac{\binom{2n}{n}}{n^k4^n}\zeta _n^{\star }(\{2\}_j)\right)\\
&=\sum _{j=0}^m\frac{(-1)^j}{2^{2j}(2m-2j+1)!}\left(\frac{\pi }{2}\right)^{2m-2j+1}\left(-\frac{2^k}{(k-1)!}\int _0^1\frac{\ln ^{k-1}\!\left(\frac{1+t}{2\sqrt{t}}\right)\operatorname{Li}_{2j}(-t)}{t}\,dt\right).
\end{align*}
Since this yields a finite lower triangular system with nonzero diagonal entries, it is invertible, and comparison of coefficients gives the desired identity.
\end{proof}

\begin{corollary} \label{corollary3.2}Let $j\in \mathbb{Z}_{\ge 0}$. Then
$$\sum _{n=1}^{\infty }\frac{\binom{2n}{n}}{n\,4^n}\zeta _n^{\star }(\{2\}_j)=2\,\eta (2j+1).$$
\end{corollary}
\begin{proof}
Taking $k=1$ in Theorem \ref{thm3.1} yields
$$\sum _{n=1}^{\infty }\frac{\binom{2n}{n}}{n\,4^n}\zeta _n^{\star }(\{2\}_j)=-2\int _0^1\frac{\operatorname{Li}_{2j}(-t)}{t}\,dt=-2\operatorname{Li}_{2j+1}(-1).$$
Since $\operatorname{Li}_s(-1)=-\eta (s)$, the claimed result follows.
\end{proof}

\begin{proposition} \label{prop3.3}Let $j\in \mathbb{Z}_{\ge 0}$. Then
$$\sum _{n=1}^{\infty }\frac{\binom{2n}{n}}{n^24^n}\zeta _n^{\star }(\{2\}_j)=2\sum _{k=0}^{2j+2}(-1)^k\eta (k)\eta (2j-k+2),$$
where we use the convention $\eta (0)=\frac{1}{2}$.
\end{proposition}
\begin{proof}
Setting $k=2$ in Theorem \ref{thm3.1}, it follows that
\begin{equation*}
\begin{aligned}
\sum _{n=1}^{\infty }\frac{\binom{2n}{n}}{n^24^n}\zeta _n^{\star }(\{2\}_j)&=-4\int _0^1\frac{\ln \!\left(\frac{1+t}{2\sqrt{t}}\right)\operatorname{Li}_{2j}(-t)}{t}\,dt\\
&=-4\int _0^1\frac{\ln (1+t)\operatorname{Li}_{2j}(-t)}{t}\,dt+2\int _0^1\frac{\ln (t)\operatorname{Li}_{2j}(-t)}{t}\,dt+4\ln (2)\int _0^1\frac{\operatorname{Li}_{2j}(-t)}{t}\,dt\\
&=-4\int _0^1\frac{\ln (1+t)\operatorname{Li}_{2j}(-t)}{t}\,dt+2\,\eta (2j+2)-4\ln (2)\eta (2j+1).
\end{aligned}\label{eq3.1}\tag{3.1}
\end{equation*}
Applying integration by parts to the remaining integral, we obtain
$$\int _0^1\frac{\ln (1+t)\operatorname{Li}_{2j}(-t)}{t}\,dt=-\eta (2)\eta (2j)+\int _0^1\frac{\operatorname{Li}_2(-t)\operatorname{Li}_{2j-1}(-t)}{t}\,dt.$$
Iterating this $2j-2$ times and using $\operatorname{Li}_1(-t)=-\ln (1+t)$ yields
\begin{equation*}
\begin{aligned}
\int _0^1\frac{\ln (1+t)\operatorname{Li}_{2j}(-t)}{t}\,dt&=\sum _{k=0}^{2j-2}(-1)^{k-1}\eta (k+2)\eta (2j-k)-\int _0^1\frac{\operatorname{Li}_{2j}(-t)\ln (1+t)}{t}\,dt\\
&=\frac{1}{2}\sum _{k=2}^{2j}(-1)^{k-1}\eta (k)\eta (2j-k+2)\\
&=\frac{1}{2}\sum _{k=0}^{2j+2}(-1)^{k-1}\eta (k)\eta (2j-k+2)+\frac{1}{2}\,\eta (2j+2)-\ln (2)\eta (2j+1).
\end{aligned}\label{eq3.2}\tag{3.2}
\end{equation*}
Substituting \eqref{eq3.2} into \eqref{eq3.1} completes the proof.
\end{proof}

\begin{theorem} \label{thm3.4}Let $j\in \mathbb{Z}_{\ge 0}, k\in \mathbb{Z}_{>0}$. Then
$$\sum _{n=1}^{\infty }\frac{4^n}{n^{k+1}\binom{2n}{n}}t_n^{\star }(\{2\}_j)=\frac{2^{k+3}}{(k-1)!}\int _0^1\frac{\ln ^{k-1}\!\left(\frac{1+t^2}{2t}\right)\operatorname{Ti}_{2j+1}(t)}{1+t^2}\,dt.$$
\end{theorem}
\begin{proof}
Applying part (ii) of Lemma \ref{lma2.2} and part (i) of Lemma \ref{lma2.3} to
$$\sum _{n=1}^{\infty }\frac{1}{n^k}\int _0^{\frac{\pi }{2}}x^{2m}\cos ^{2n-1}(x)\,dx=\int _0^{\frac{\pi }{2}}x^{2m}\frac{\operatorname{Li}_k(\cos ^2(x))}{\cos (x)}\,dx,$$
we obtain
\begin{align*}
&\frac{(2m)!}{2}\sum _{n=1}^{\infty }\frac{4^n}{n^{k+1}\binom{2n}{n}}\sum _{j=0}^m\frac{(-1)^jt_n^{\star }(\{2\}_j)}{(2m-2j)!}\left(\frac{\pi }{2}\right)^{2m-2j}\\
&=\frac{2^k}{(k-1)!}\sum _{n=1}^{\infty }\left(\int _0^1t^{n-1}\ln ^{k-1}\!\left(\frac{1+t}{2\sqrt{t}}\right)\,dt\right)\int _0^{\frac{\pi }{2}}x^{2m}\frac{(-1)^{n-1}+\cos (2nx)}{\cos (x)}\,dx.
\end{align*}
Using part (iii) of Lemma \ref{lma2.1} gives
\begin{equation*}
\begin{aligned}
&\frac{(2m)!}{2}\sum _{j=0}^m\frac{(-1)^j}{(2m-2j)!}\left(\frac{\pi }{2}\right)^{2m-2j}\sum _{n=1}^{\infty }\frac{4^n}{n^{k+1}\binom{2n}{n}}t_n^{\star }(\{2\}_j)\\
&=(2m)!\frac{2^{k+1}}{(k-1)!}\sum _{n=1}^{\infty }\left(\int _0^1t^{n-1}\ln ^{k-1}\!\left(\frac{1+t}{2\sqrt{t}}\right)\,dt\right)\sum _{j=0}^m\frac{(-1)^{j+n-1}O_n^{(2j+1)}}{(2m-2j)!}\left(\frac{\pi }{2}\right)^{2m-2j}\\
&\overset{t\mapsto t^2}{=}(2m)!\frac{2^{k+2}}{(k-1)!}\sum _{j=0}^m\frac{(-1)^j}{(2m-2j)!}\left(\frac{\pi }{2}\right)^{2m-2j}\int _0^1\left(\sum _{n=1}^{\infty }(-1)^{n-1}O_n^{(2j+1)}t^{2n-1}\right)\ln ^{k-1}\!\left(\frac{1+t^2}{2t}\right)\,dt.
\end{aligned}\label{eq3.3}\tag{3.3}
\end{equation*}
Furthermore, for $|t|<1$, it follows that
\begin{equation*}
\begin{aligned}
\sum _{n=1}^{\infty }(-1)^{n-1}O_n^{(2j+1)}t^{2n-1}&=\sum _{n=1}^{\infty }(-1)^{n-1}t^{2n-1}\sum _{k=1}^n\frac{1}{(2k-1)^{2j+1}}\\
&=\sum _{k=1}^{\infty }\frac{1}{(2k-1)^{2j+1}}\sum _{n=k}^{\infty }(-1)^{n-1}t^{2n-1}\\
&=\frac{1}{1+t^2}\sum _{k=1}^{\infty }\frac{(-1)^{k-1}t^{2k-1}}{(2k-1)^{2j+1}}=\frac{\operatorname{Ti}_{2j+1}(t)}{1+t^2}.
\end{aligned}\label{eq3.4}\tag{3.4}
\end{equation*}
Substituting \eqref{eq3.4} into \eqref{eq3.3}, we deduce that
\begin{align*}
&\sum _{j=0}^m\frac{(-1)^j}{(2m-2j)!}\left(\frac{\pi }{2}\right)^{2m-2j}\left(\sum _{n=1}^{\infty }\frac{4^n}{n^{k+1}\binom{2n}{n}}t_n^{\star }(\{2\}_j)\right)\\
&=\sum _{j=0}^m\frac{(-1)^j}{(2m-2j)!}\left(\frac{\pi }{2}\right)^{2m-2j}\left(\frac{2^{k+3}}{(k-1)!}\int _0^1\frac{\ln ^{k-1}\!\left(\frac{1+t^2}{2t}\right)\operatorname{Ti}_{2j+1}(t)}{1+t^2}\,dt\right).
\end{align*}
As in the proof of Theorem \ref{thm3.1}, the identity follows by comparing coefficients in the resulting lower triangular system.
\end{proof}
The following proposition is the main result of \cite{preprint2}, where it first appears and is obtained using less general identities.
\begin{proposition} \label{prop3.5}Let $j\in \mathbb{Z}_{\ge 0}$. Then
$$\sum _{n=1}^{\infty }\frac{4^n}{n^2\binom{2n}{n}}t_n^{\star}(\{2\}_j)=8\sum _{k=0}^{2j}(-1)^k\beta (k+1)\beta (2j-k+1).$$
\end{proposition}
\begin{proof}
Setting $k=1$ in Theorem \ref{thm3.4}, we obtain
$$\sum _{n=1}^{\infty }\frac{4^n}{n^2\binom{2n}{n}}t_n^{\star }(\{2\}_j)=16\int _0^1\frac{\operatorname{Ti}_{2j+1}(t)}{1+t^2}\,dt.$$
As in the proof of part (iii) of Lemma 2 in \cite{preprint2}, integration by parts gives
$$\int _0^1\frac{\operatorname{Ti}_{2j+1}(t)}{1+t^2}\,dt=\beta (1)\beta (2j+1)-\int _0^1\frac{\arctan (t)\operatorname{Ti}_{2j}(t)}{t}\,dt.$$
Iterating this $2j$ times and noting that $\operatorname{Ti}_0(t)=\frac{t}{1+t^2}$, we obtain
$$\int _0^1\frac{\operatorname{Ti}_{2j+1}(t)}{1+t^2}\,dt=\sum _{k=0}^{2j}(-1)^k\beta (k+1)\beta (2j-k+1)-\int _0^1\frac{\operatorname{Ti}_{2j+1}(t)}{1+t^2}\,dt.$$
Since the final integral coincides with the original one, rearranging completes the proof.
\end{proof}

\begin{theorem} \label{thm3.6}Let $j,k\in \mathbb{Z}_{>0}$. Then
$$\sum _{n=1}^{\infty }\frac{\binom{2n}{n}}{n^k4^n}\sum _{n\ge n_1\ge \cdots \ge n_j\ge 1}\frac{4^{n_j}}{\binom{2n_j}{n_j}}\prod _{i=1}^j\frac{1}{n_i^2}=\frac{2^{k+1}}{(k-1)!}\int _0^1\frac{\ln ^{k-1}\!\left(\frac{1+t}{2\sqrt{t}}\right)\chi _{2j}(t)}{t}\,dt.$$
\end{theorem}
\begin{proof}
From
$$\sum _{n=1}^{\infty }\frac{1}{n^k}\int _0^{\frac{\pi }{2}}x^{2m-1}\cos ^{2n}(x)\,dx=\int _0^{\frac{\pi }{2}}x^{2m-1}\operatorname{Li}_k(\cos ^2(x))\,dx,$$
and parts (iii) and (i) of Lemmas \ref{lma2.2} and \ref{lma2.3}, respectively, we obtain
\begin{align*}
&(2m-1)!\sum _{n=1}^{\infty }\frac{\binom{2n}{n}}{n^k4^n}\left(\sum _{j=0}^{m-1}\frac{(-1)^j\zeta _n^{\star }(\{2\}_j)}{2^{2j}(2m-2j)!}\left(\frac{\pi }{2}\right)^{2m-2j}+\frac{(-1)^m}{4^m}\sum _{n\ge n_1\ge \cdots \ge n_m\ge 1}\frac{4^{n_m}}{\binom{2n_m}{n_m}}\prod _{i=1}^m\frac{1}{n_i^2}\right)\\
&=\frac{2^k}{(k-1)!}\sum _{n=1}^{\infty }\left(\int _0^1t^{n-1}\ln ^{k-1}\!\left(\frac{1+t}{2\sqrt{t}}\right)\,dt\right)\int _0^{\frac{\pi }{2}}x^{2m-1}\left((-1)^{n-1}+\cos (2nx)\right)\,dx.
\end{align*}
From part (ii) of Lemma \ref{lma2.1}, it follows that
\begin{equation*}
\begin{aligned}
&(2m-1)!\left(\sum _{j=0}^{m-1}\frac{(-1)^j}{2^{2j}(2m-2j)!}\left(\frac{\pi }{2}\right)^{2m-2j}\sum _{n=1}^{\infty }\frac{\binom{2n}{n}}{n^k4^n}\zeta _n^{\star }(\{2\}_j)\right.\\
&\qquad\qquad\left.+\frac{(-1)^m}{4^m}\sum _{n=1}^{\infty }\frac{\binom{2n}{n}}{n^k4^n}\sum _{n\ge n_1\ge \cdots \ge n_m\ge 1}\frac{4^{n_m}}{\binom{2n_m}{n_m}}\prod _{i=1}^m\frac{1}{n_i^2}\right)\\
&=(2m-1)!\frac{2^k}{(k-1)!}\left(\sum _{j=0}^{m-1}\frac{(-1)^{j-1}}{2^{2j}(2m-2j)!}\left(\frac{\pi }{2}\right)^{2m-2j}\int _0^1\left(\sum _{n=1}^{\infty }\frac{(-t)^n}{n^{2j}}\right)\frac{\ln ^{k-1}\!\left(\frac{1+t}{2\sqrt{t}}\right)}{t}\,dt\right.\\
&\qquad\qquad\left.+\frac{(-1)^m}{4^m}\int _0^1\left(\sum _{n=1}^{\infty }\frac{t^n-(-t)^n}{n^{2m}}\right)\frac{\ln ^{k-1}\!\left(\frac{1+t}{2\sqrt{t}}\right)}{t}\,dt\right)\\
&=(2m-1)!\frac{2^k}{(k-1)!}\left(\sum _{j=0}^{m-1}\frac{(-1)^{j-1}}{2^{2j}(2m-2j)!}\left(\frac{\pi }{2}\right)^{2m-2j}\int _0^1\frac{\ln ^{k-1}\!\left(\frac{1+t}{2\sqrt{t}}\right)\operatorname{Li}_{2j}(-t)}{t}\,dt\right.\\
&\qquad\qquad\left.+\frac{(-1)^m}{4^m}\int _0^1\frac{\ln ^{k-1}\!\left(\frac{1+t}{2\sqrt{t}}\right)\left(\operatorname{Li}_{2m}(t)-\operatorname{Li}_{2m}(-t)\right)}{t}\,dt\right).
\end{aligned}\label{eq3.5}\tag{3.5}
\end{equation*}
By Theorem \ref{thm3.1} together with $\chi _s(x)=\frac{1}{2}\left(\operatorname{Li}_s(x)-\operatorname{Li}_s(-x)\right)$, the right-hand side simplifies to
\begin{align*}
&(2m-1)!\left(\sum _{j=0}^{m-1}\frac{(-1)^j}{2^{2j}(2m-2j)!}\left(\frac{\pi }{2}\right)^{2m-2j}\sum _{n=1}^{\infty }\frac{\binom{2n}{n}}{n^k4^n}\zeta _n^{\star }(\{2\}_j)\right.\\
&\qquad\qquad\left.+\frac{(-1)^m}{4^m}\frac{2^{k+1}}{(k-1)!}\int _0^1\frac{\ln ^{k-1}\!\left(\frac{1+t}{2\sqrt{t}}\right)\chi _{2m}(t)}{t}\,dt\right).
\end{align*}
Therefore, cancelling the identical expressions on both sides of \eqref{eq3.5} and relabeling $m$ as $j$ yields the desired identity.
\end{proof}

\begin{corollary} \label{corollary3.7}Let $j\in \mathbb{Z}_{>0}$. Then
$$\sum _{n=1}^{\infty }\frac{\binom{2n}{n}}{n\,4^n}\sum _{n\ge n_1\ge \cdots \ge n_j\ge 1}\frac{4^{n_j}}{\binom{2n_j}{n_j}}\prod _{i=1}^j\frac{1}{n_i^2}=4\,\lambda (2j+1).$$
\end{corollary}
\begin{proof}
Taking $k=1$ in Theorem \ref{thm3.6} gives
$$\sum _{n=1}^{\infty }\frac{\binom{2n}{n}}{n\,4^n}\sum _{n\ge n_1\ge \cdots \ge n_j\ge 1}\frac{4^{n_j}}{\binom{2n_j}{n_j}}\prod _{i=1}^j\frac{1}{n_i^2}=4\int _0^1\frac{\chi _{2j}(t)}{t}\,dt=4\,\chi _{2j+1}(1).$$
The claim now follows from $\chi _s(1)=\lambda (s)$.
\end{proof}

\begin{theorem} \label{thm3.8}Let $j,k\in \mathbb{Z}_{>0}$. Then
$$\sum _{n=1}^{\infty }\frac{4^n}{n^{k+1}\binom{2n}{n}}\sum _{n\ge n_1\ge \cdots \ge n_{j-1}>n_j\ge 0}\frac{\binom{2n_j}{n_j}}{(2n_j+1)4^{n_j}}\prod _{i=1}^{j-1}\frac{1}{(2n_i-1)^2}=\frac{2^{k+3}}{(k-1)!}\int _0^1\frac{\ln ^{k-1}\!\left(\frac{1+t^2}{2t}\right)\chi _{2j}(t)}{1+t^2}\,dt.$$
\end{theorem}
\begin{proof}
Using
$$\sum _{n=1}^{\infty }\frac{1}{n^k}\int _0^{\frac{\pi }{2}}x^{2m-1}\sin ^{2n-1}(x)\,dx=\int _0^{\frac{\pi }{2}}x^{2m-1}\frac{\operatorname{Li}_k(\sin ^2(x))}{\sin (x)}\,dx,$$
together with part (iv) of Lemma \ref{lma2.2} and part (ii) of Lemma \ref{lma2.3}, we derive
\begin{align*}
&\frac{(2m-1)!}{2}\sum _{n=1}^{\infty }\frac{4^n}{n^{k+1}\binom{2n}{n}}\sum _{j=1}^m\frac{(-1)^{j-1}}{(2m-2j)!}\left(\frac{\pi }{2}\right)^{2m-2j}\sum _{n\ge n_1\ge \cdots \ge n_{j-1}>n_j\ge 0}\frac{\binom{2n_j}{n_j}}{(2n_j+1)4^{n_j}}\prod _{i=1}^{j-1}\frac{1}{(2n_i-1)^2}\\
&=\frac{2^k}{(k-1)!}\sum _{n=1}^{\infty }(-1)^{n-1}\left(\int _0^1t^{n-1}\ln ^{k-1}\!\left(\frac{1+t}{2\sqrt{t}}\right)\,dt\right)\int _0^{\frac{\pi }{2}}x^{2m-1}\frac{1-\cos (2nx)}{\sin (x)}\,dx.
\end{align*}
Applying part (iv) of Lemma \ref{lma2.1}, we obtain
\begin{equation*}
\begin{aligned}
&\frac{(2m-1)!}{2}\sum _{j=1}^m\frac{(-1)^{j-1}}{(2m-2j)!}\left(\frac{\pi }{2}\right)^{2m-2j}\sum _{n=1}^{\infty }\frac{4^n}{n^{k+1}\binom{2n}{n}}\sum _{n\ge n_1\ge \cdots \ge n_{j-1}>n_j\ge 0}\frac{\binom{2n_j}{n_j}}{(2n_j+1)4^{n_j}}\prod _{i=1}^{j-1}\frac{1}{(2n_i-1)^2}\\
&=(2m-1)!\frac{2^{k+1}}{(k-1)!}\sum _{n=1}^{\infty }(-1)^{n-1}\left(\int _0^1t^{n-1}\ln ^{k-1}\!\left(\frac{1+t}{2\sqrt{t}}\right)\,dt\right)\sum _{j=1}^m\frac{(-1)^{j-1}\overline{O}_n^{(2j)}}{(2m-2j)!}\left(\frac{\pi }{2}\right)^{2m-2j}\\
&\overset{t\mapsto t^2}{=}(2m-1)!\frac{2^{k+2}}{(k-1)!}\sum _{j=1}^m\frac{(-1)^{j-1}}{(2m-2j)!}\left(\frac{\pi }{2}\right)^{2m-2j}\int _0^1\left(\sum _{n=1}^{\infty }(-1)^{n-1}\overline{O}_n^{(2j)}t^{2n-1}\right)\ln ^{k-1}\!\left(\frac{1+t^2}{2t}\right)\,dt.
\end{aligned}\label{eq3.6}\tag{3.6}
\end{equation*}
In addition, for $|t|<1$, we have that
\begin{equation*}
\begin{aligned}
\sum _{n=1}^{\infty }(-1)^{n-1}\overline{O}_n^{(2j)}t^{2n-1}&=\sum _{n=1}^{\infty }(-1)^{n-1}t^{2n-1}\sum _{k=1}^n\frac{(-1)^{k-1}}{(2k-1)^{2j}}\\
&=\sum _{k=1}^{\infty }\frac{(-1)^{k-1}}{(2k-1)^{2j}}\sum _{n=k}^{\infty }(-1)^{n-1}t^{2n-1}\\
&=\frac{1}{1+t^2}\sum _{k=1}^{\infty }\frac{t^{2k-1}}{(2k-1)^{2j}}=\frac{\chi _{2j}(t)}{1+t^2}.
\end{aligned}\label{eq3.7}\tag{3.7}
\end{equation*}
Thus, by substituting \eqref{eq3.7} into \eqref{eq3.6}, we arrive at
\begin{align*}
&\sum _{j=1}^m\frac{(-1)^{j-1}}{(2m-2j)!}\left(\frac{\pi }{2}\right)^{2m-2j}\left(\sum _{n=1}^{\infty }\frac{4^n}{n^{k+1}\binom{2n}{n}}\sum _{n\ge n_1\ge \cdots \ge n_{j-1}>n_j\ge 0}\frac{\binom{2n_j}{n_j}}{(2n_j+1)4^{n_j}}\prod _{i=1}^{j-1}\frac{1}{(2n_i-1)^2}\right)\\
&=\sum _{j=1}^m\frac{(-1)^{j-1}}{(2m-2j)!}\left(\frac{\pi }{2}\right)^{2m-2j}\left(\frac{2^{k+3}}{(k-1)!}\int _0^1\frac{\ln ^{k-1}\!\left(\frac{1+t^2}{2t}\right)\chi _{2j}(t)}{1+t^2}\,dt\right).
\end{align*}
The desired identity follows by comparing coefficients in the resulting lower triangular structure.
\end{proof}

\begin{theorem} \label{thm3.9}Let $j\in \mathbb{Z}_{\ge 0}, k\in \mathbb{Z}_{>0}$. Then
$$\sum _{n=0}^{\infty }\frac{4^n}{(2n+1)^{k+1}\binom{2n}{n}}t_{n+1}^{\star }(\{2\}_j)=\frac{2}{(k-1)!}\int _0^1\frac{\ln ^{k-1}\!\left(\frac{1+t^2}{2t}\right)\operatorname{Ti}_{2j+1}(t)}{t}\,dt.$$
\end{theorem}
\begin{proof}
Applying the substitution $n\mapsto n+1$ and the identity $\binom{2n+2}{n+1}=\frac{2(2n+1)}{n+1}\binom{2n}{n}$ to part (ii) of Lemma \ref{lma2.2}, it follows that
$$\int _0^{\frac{\pi }{2}}x^{2m}\cos ^{2n+1}(x)\,dx=(2m)!\frac{4^n}{(2n+1)\binom{2n}{n}}\sum _{j=0}^m\frac{(-1)^jt_{n+1}^{\star }(\{2\}_j)}{(2m-2j)!}\left(\frac{\pi }{2}\right)^{2m-2j}.$$
Multiplying this by $\frac{1}{(2n+1)^k}$ and summing $n$ on both sides from $0$ to $\infty $ yields
\begin{equation*}
(2m)!\sum _{n=0}^{\infty }\frac{4^n}{(2n+1)^{k+1}\binom{2n}{n}}\sum _{j=0}^m\frac{(-1)^jt_{n+1}^{\star }(\{2\}_j)}{(2m-2j)!}\left(\frac{\pi }{2}\right)^{2m-2j}=\int _0^{\frac{\pi }{2}}x^{2m}\chi _k(\cos (x))\,dx.\label{eq3.8}\tag{3.8}
\end{equation*}
Furthermore, by part (iii) of Lemma \ref{lma2.3} and the identity
$$\int _0^{\frac{\pi }{2}}x^{2m}\cos ((2n-1)x)\,dx=(2m)!\sum _{j=0}^m\frac{(-1)^{j+n-1}}{(2n-1)^{2j+1}(2m-2j)!}\left(\frac{\pi }{2}\right)^{2m-2j},$$
the right-hand side of \eqref{eq3.8} becomes
\begin{equation*}
\begin{aligned}
&\int _0^{\frac{\pi }{2}}x^{2m}\chi _k(\cos (x))\,dx\\
&=\frac{2}{(k-1)!}\sum _{n=1}^{\infty }\left(\int _0^1t^{2n-2}\ln ^{k-1}\!\left(\frac{1+t^2}{2t}\right)\,dt\right)\int _0^{\frac{\pi }{2}}x^{2m}\cos ((2n-1)x)\,dx\\
&=(2m)!\frac{2}{(k-1)!}\sum _{j=0}^m\frac{(-1)^j}{(2m-2j)!}\left(\frac{\pi }{2}\right)^{2m-2j}\int _0^1\frac{\ln ^{k-1}\!\left(\frac{1+t^2}{2t}\right)\operatorname{Ti}_{2j+1}(t)}{t}\,dt.
\end{aligned}\label{eq3.9}\tag{3.9}
\end{equation*}
Therefore, applying \eqref{eq3.9} to \eqref{eq3.8}, we derive
\begin{align*}
&\sum _{j=0}^m\frac{(-1)^j}{(2m-2j)!}\left(\frac{\pi }{2}\right)^{2m-2j}\left(\sum _{n=0}^{\infty }\frac{4^n}{(2n+1)^{k+1}\binom{2n}{n}}t_{n+1}^{\star }(\{2\}_j)\right)\\
&=\sum _{j=0}^m\frac{(-1)^j}{(2m-2j)!}\left(\frac{\pi }{2}\right)^{2m-2j}\left(\frac{2}{(k-1)!}\int _0^1\frac{\ln ^{k-1}\!\left(\frac{1+t^2}{2t}\right)\operatorname{Ti}_{2j+1}(t)}{t}\,dt\right).
\end{align*}
By comparing coefficients, we establish the stated identity.
\end{proof}
The next corollary appears as the second result in \cite[Example~2.8]{preprint3}
\begin{corollary} \label{corollary3.10}Let $j\in \mathbb{Z}_{\ge 0}$. Then
$$\sum _{n=0}^{\infty }\frac{4^n}{(2n+1)^2\binom{2n}{n}}t_{n+1}^{\star }(\{2\}_j)=2\,\beta (2j+2).$$
\end{corollary}
\begin{proof}
Letting $k=1$ in Theorem \ref{thm3.9} yields
$$\sum _{n=0}^{\infty }\frac{4^n}{(2n+1)^2\binom{2n}{n}}t_{n+1}^{\star }(\{2\}_j)=2\int _0^1\frac{\operatorname{Ti}_{2j+1}(t)}{t}\,dt=2\operatorname{Ti}_{2j+2}(1).$$
Using $\operatorname{Ti}_s(1)=\beta (s)$ gives the claimed result.
\end{proof}

\begin{theorem} \label{thm3.11}Let $j,k\in \mathbb{Z}_{>0}$. Then
$$\sum _{n=0}^{\infty }\frac{4^n}{(2n+1)^{k+1}\binom{2n}{n}}\sum _{n\ge n_1\ge \dots \ge n_j\ge 0}\frac{\binom{2n_j}{n_j}}{(2n_j+1)4^{n_j}}\prod _{i=1}^{j-1}\frac{1}{(2n_i+1)^2}=\frac{2}{(k-1)!}\int _0^1\frac{\ln ^{k-1}\!\left(\frac{1+t^2}{2t}\right)\chi _{2j}(t)}{t}\,dt.$$
\end{theorem}
\begin{proof}
Replacing $n$ by $n+1$ in part (iv) of Lemma \ref{lma2.2}, and using the identity $\binom{2n+2}{n+1}=\frac{2(2n+1)}{n+1}\binom{2n}{n}$, we have
\begin{align*}
&\int _0^{\frac{\pi }{2}}x^{2m-1}\sin ^{2n+1}(x)\,dx\\
&=(2m-1)!\frac{4^n}{(2n+1)\binom{2n}{n}}\sum _{j=1}^m\frac{(-1)^{j-1}}{(2m-2j)!}\left(\frac{\pi }{2}\right)^{2m-2j}\sum _{n\ge n_1\ge \dots \ge n_j\ge 0}\frac{\binom{2n_j}{n_j}}{(2n_j+1)4^{n_j}}\prod _{i=1}^{j-1}\frac{1}{(2n_i+1)^2}.
\end{align*}
Multiplying both sides by $\frac{1}{(2n+1)^k}$ and summing from $n=0$ to $\infty $, we obtain
\begin{equation*}
\begin{aligned}
&\int _0^{\frac{\pi }{2}}x^{2m-1}\chi _k(\sin (x))\,dx\\
&=(2m-1)!\sum _{n=0}^{\infty }\frac{4^n}{(2n+1)^{k+1}\binom{2n}{n}}\sum _{j=1}^m\frac{(-1)^{j-1}}{(2m-2j)!}\left(\frac{\pi }{2}\right)^{2m-2j}\sum _{n\ge n_1\ge \dots \ge n_j\ge 0}\frac{\binom{2n_j}{n_j}}{(2n_j+1)4^{n_j}}\prod _{i=1}^{j-1}\frac{1}{(2n_i+1)^2}.
\end{aligned}\label{eq3.10}\tag{3.10}
\end{equation*}
In addition, from part (iv) of Lemma \ref{lma2.3} and the result
$$\int _0^{\frac{\pi }{2}}x^{2m-1}\sin ((2n-1)x)\,dx=(2m-1)!\sum _{j=1}^m\frac{(-1)^{j+n}}{(2n-1)^{2j}(2m-2j)!}\left(\frac{\pi }{2}\right)^{2m-2j},$$
it follows that
\begin{equation*}
\begin{aligned}
&\int _0^{\frac{\pi }{2}}x^{2m-1}\chi _k(\sin (x))\,dx\\
&=\frac{2}{(k-1)!}\sum _{n=1}^{\infty }(-1)^{n-1}\left(\int _0^1t^{2n-2}\ln ^{k-1}\!\left(\frac{1+t^2}{2t}\right)\,dt\right)\int _0^{\frac{\pi }{2}}x^{2m-1}\sin ((2n-1)x)\,dx\\
&=(2m-1)!\frac{2}{(k-1)!}\sum _{j=1}^m\frac{(-1)^{j-1}}{(2m-2j)!}\left(\frac{\pi }{2}\right)^{2m-2j}\int _0^1\frac{\ln ^{k-1}\!\left(\frac{1+t^2}{2t}\right)\chi _{2j}(t)}{t}\,dt.
\end{aligned}\label{eq3.11}\tag{3.11}
\end{equation*}
Substituting \eqref{eq3.11} in \eqref{eq3.10} yields
\begin{align*}
&\sum _{j=1}^m\frac{(-1)^{j-1}}{(2m-2j)!}\left(\frac{\pi }{2}\right)^{2m-2j}\left(\sum _{n=0}^{\infty }\frac{4^n}{(2n+1)^{k+1}\binom{2n}{n}}\sum _{n\ge n_1\ge \dots \ge n_j\ge 0}\frac{\binom{2n_j}{n_j}}{(2n_j+1)4^{n_j}}\prod _{i=1}^{j-1}\frac{1}{(2n_i+1)^2}\right)\\
&=\sum _{j=1}^m\frac{(-1)^{j-1}}{(2m-2j)!}\left(\frac{\pi }{2}\right)^{2m-2j}\left(\frac{2}{(k-1)!}\int _0^1\frac{\ln ^{k-1}\!\left(\frac{1+t^2}{2t}\right)\chi _{2j}(t)}{t}\,dt\right).
\end{align*}
Upon comparing coefficients, the desired result follows.
\end{proof}
The following corollary recovers Theorem 2.5 in \cite{journal2}.
\begin{corollary} \label{corollary3.12}Let $j\in \mathbb{Z}_{>0}$. Then
$$\sum _{n=0}^{\infty }\frac{4^n}{(2n+1)^2\binom{2n}{n}}\sum _{n\ge n_1\ge \dots \ge n_j\ge 0}\frac{\binom{2n_j}{n_j}}{(2n_j+1)4^{n_j}}\prod _{i=1}^{j-1}\frac{1}{(2n_i+1)^2}=2\,\lambda (2j+1).$$
\end{corollary}
\begin{proof}
Setting $k=1$ in Theorem \ref{thm3.11}, we have
$$\sum _{n=0}^{\infty }\frac{4^n}{(2n+1)^2\binom{2n}{n}}\sum _{n\ge n_1\ge \dots \ge n_j\ge 0}\frac{\binom{2n_j}{n_j}}{(2n_j+1)4^{n_j}}\prod _{i=1}^{j-1}\frac{1}{(2n_i+1)^2}=2\int _0^1\frac{\chi _{2j}(t)}{t}\,dt=2\,\chi _{2j+1}(1).$$
From $\chi _s(1)=\lambda (s)$, we obtain the stated identity.
\end{proof}
\newpage
\section{Concluding Remarks}
In this work, we have established integral representations for six families of multiple Ap\'ery-like series using elementary methods, namely repeated applications of integration by parts to trigonometric integrals together with Fourier expansions for the polylogarithm and the Legendre chi function with trigonometric arguments. The combination of these tools led to the recovery of several known evaluations and yielded a new identity expressing a class of such series as finite alternating sums of products of Dirichlet eta values.

A natural direction for future work would be the application of the techniques in Lemma \ref{lma2.2} to trigonometric integrals over alternative ranges. The restriction to $[0,\frac{\pi}{2}]$ plays a key role in yielding the central binomial coefficients $4^n/n\binom{2n}{n}$ and $\binom{2n}{n}/4^n$, and the associated nested harmonic structures in a tractable form. Replacing this interval by $[0,\frac{\pi}{3}]$ or $[0,\frac{\pi}{4}]$ is of particular interest, since the endpoint values $\cos (\frac{\pi }{3})=\frac{1}{2}$ and $\cos (\frac{\pi }{4})=\frac{1}{\sqrt{2}}$ lead to modified integral evaluations, for example
$$\int _0^{\frac{\pi }{3}}\cos ^{2n}(x)\,dx=\frac{\binom{2n}{n}}{4^n}\left(\frac{\pi }{3}+\frac{\sqrt{3}}{2}\sum _{n_1=1}^n\frac{1}{n_1\binom{2n_1}{n_1}}\right),\quad \int _0^{\frac{\pi }{4}}\cos ^{2n}(x)\,dx=\frac{\binom{2n}{n}}{4^n}\left(\frac{\pi }{4}+\frac{1}{2}\sum _{n_1=1}^n\frac{2^{n_1}}{n_1\binom{2n_1}{n_1}}\right),$$
which suggest that related multiple Apéry-like series with different combinatorial weights may admit similar integral representations.

Moreover, it remains an open question whether such representations of a similar type can be derived using Fourier expansions not considered in the present work, perhaps involving other special functions, and whether analogous closed-form evaluations can be obtained for further values of the parameter $k$ in the remaining families, where such formulas are more difficult to establish.


\begin{thebibliography}{999}
\bibitem{book5}  C.~I.~V\u{a}lean, \textit{(Almost) Impossible Integrals, Sums, and Series}, Springer, Cham, 2019.
\bibitem{book1}  C.~I.~V\u{a}lean, \textit{More (Almost) Impossible Integrals, Sums, and Series}, Springer, Cham, 2023.
\bibitem{preprint3} C.~Xu and J.~Zhao, Ap\'ery-Type Series with Summation Indices of Mixed Parities and Colored Multiple Zeta Values, I, \textit{arXiv preprint} arXiv:2202.06195v2, 2022.
\bibitem{journal5} C.~Xu and J.~Zhao, Ap\'ery-type series and colored multiple zeta values, \textit{Advances in Applied Mathematics}, \textbf{153}:102610, 2024.
\bibitem{preprint1} C.~Xu, On the proof of the Gen{\v c}ev--Rucki conjecture for multiple Ap\'ery-like series, \textit{arXiv preprint} arXiv:2510.09052, 2025.
\bibitem{journal9} D.~H.~Lehmer, Interesting Series Involving the Central Binomial Coefficient, \textit{The American Mathematical Monthly}, \textbf{92}(7):449--457, 1985.
\bibitem{journal11} I.~J.~Zucker, On the series $\sum _{k=1}^{\infty }\binom{2k}{k}^{-1}k^{-n}$ and related sums, \textit{Journal of Number Theory}, \textbf{20}(1):92--102, 1985.
\bibitem{book3}  I.~S.~Gradshteyn and I.~M.~Ryzhik, \textit{Table of Integrals, Series, and Products}, 8th edition, Academic Press, New York, 2014.
\bibitem{journal10} J.~M.~Borwein,~D.~J.~Broadhurst and J.~Kamnitzer, Central Binomial Sums, Multiple Clausen Values, and Zeta Values, \textit{Experimental Mathematics}, \textbf{10}(1):25--34, 2001.
\bibitem{journal7} J.~A.~G.~Layja, Evaluating six Ap\'ery-like series of weight 5, \textit{Integral Transforms and Special Functions}, \textbf{37}(4):249--265, 2026.
\bibitem{preprint2} J.~A.~G.~Layja, On the Evaluation of Ap\'ery-Like Series Involving Multiple $t$-Harmonic Star Sums, \textit{arXiv preprint} arXiv:2601.20027, 2026.
\bibitem{book4}  L.~Lewin, \textit{Polylogarithms and Associated Functions}, North Holland, New York, 1981.
\bibitem{journal3} M.~E.~Hoffman, Multiple harmonic series, \textit{Pacific Journal of Mathematics}, \textbf{152}(2):275--290, 1992.
\bibitem{journal1} M.~E.~Hoffman, An odd variant of multiple zeta values, \textit{Communications in Number Theory and Physics}, \textbf{13}(3):529--567, 2019.
\bibitem{journal8} M.~Cantarini and J.~D'Aurizio, On the interplay between hypergeometric series, Fourier--Legendre expansions and Euler sums, \textit{Bollettino dell'Unione Matematica Italiana}, \textbf{12}:623--656, 2019.
\bibitem{journal2} M.~Gen\v{c}ev and P.~Rucki, On a class of multiple Ap\'ery-like series and their reduction, \textit{Mediterranean Journal of Mathematics}, \textbf{22}:195, 2025.
\bibitem{journal4} R.~Ap\'ery, Irrationalit\'e de $\zeta (2)$ et $\zeta (3)$, \textit{Ast\'erisque}, \textbf{61}:11--13, 1979.
\bibitem{journal6} W.~Chu, Alternating series of Ap\'ery-type for the Riemann zeta function, \textit{Contributions to Discrete Mathematics}, \textbf{15}(3):108--116, 2020.
\bibitem{journal12} X.~Chen and W.~Wang, Ap\'ery-type series via colored multiple zeta values and Fourier-Legendre series expansions, \textit{Journal of Symbolic Computation}, \textbf{134}:102508, 2026.
\end{thebibliography}
\end{document}